\documentclass[a4paper,11pt]{article}
\usepackage{amsmath, amsthm, amsfonts, amssymb}
\usepackage{url}
\usepackage{mathptmx}       % selects Times Roman as basic font
\usepackage{helvet}         % selects Helvetica as sans-seriffont
\usepackage{courier}        % selects Courier as typewriter font
\usepackage{type1cm}        % activate if the above 3 fonts are
\usepackage{makeidx}         % allows index generation
\usepackage[bottom]{footmisc}% places footnotes at page bottom

\def\E{{\mathbb E }}
\def\P{{\mathbb P}}
 \numberwithin{equation}{section}
\newcommand{\R}{\mathbb R}
\newtheorem{theorem}{Theorem}[section]

\newtheorem{corollary}[theorem]{Corollary}
\newtheorem{remark}[theorem]{\it \rmfamily Remark}

\newtheorem{example}[theorem]{\it \rmfamily Examples}

\newtheorem{lemma}[theorem]{Lemma}
\newtheorem{hypothesis}{\it \rmfamily Hypothesis}
\newtheorem{proposition}[theorem]{Proposition}

\makeindex             % used for the subject index
\begin{document}

\title{$L^p$-parabolic regularity and non-degenerate 
Ornstein-Uhlenbeck type operators}

\author{ Enrico Priola \footnote{
Research supported by   PRIN project 2010MXMAJR.} 
%and GNAMPA
%project “Semigruppi markoviani su algebre non commutative”, and YP the sup-
%port of ANR project “HAM-MARK”, n◦ ANR-09-BLAN-0098-01.
%\\ 
%\small 
\\ \\
{Dipartimento di Matematica ``Giuseppe Peano'', } 
\\
{Universit\`a di Torino,}
\\ {via Carlo Alberto 10, 10123 Torino 
(Italy)}
\\ 
{{enrico.priola@unito.it}}
}

%\and Name of Second Author \at Name, Address of Institute
%\email{name@email.address}}
\date{}
\maketitle

\abstract{
%In the first part of the paper we consider
 We prove 
$L^p$-parabolic a-priori estimates for $\partial_t u + \sum_{i,j=1}^d
c_{ij}(t)\partial_{x_i x_j}^2 u = f $ on $\R^{d+1}$ when the coefficients
$c_{ij}$  are  locally bounded functions on $\R$. We slightly generalize
the usual parabolicity assumption and show that still  $L^p$-estimates hold for  the
second spatial derivatives of $u$. We also investigate  the dependence of
the constant appearing in such estimates from the parabolicity constant. 
%Then 
Finally we
extend our estimates to parabolic equations involving  non-degenerate
Ornstein-Uhlenbeck type operators. 
%
%provide a new proof of parabolic $L^p$-estimates for non-degenerate
%Ornstein-Uhlnebeck type operators. This proof shows the correct dependence %of
%the constant from the parabolicity constant and also slightly %improve %the
%regularity of the diffusion matrix. 
}



%\large{

\section{Introduction and basic notations}
\label{sec:1}

In this paper  we deal with global  a-priori $L^p$-estimates
for solutions $u$ to  second order parabolic equations like 
 \begin{equation}
 u_t(t,x) + \sum_{i,j=1}^d c_{ij}(t) u_{x_i x_j}(t,x) 
 = f(t,x),\;\;\;
 (t,x) \in \R^{d+1},
 \label{basic}
\end{equation} 
$d \ge 1$, with locally bounded coefficients $c_{ij}(t)$. 
Here $u_t$ and $u_{x_i x_j}$ denote respectively the first partial
derivative with respect to $t$ and the second partial derivative with respect to
 $x_i $ and $x_j$. 
%Equation like \eqref{basic}  are natural generalization of the classical heat
%equation.
% which is obtained when  $c_{ij}(t) =0$, $i \not = j$, and
% $c_{ii}(t) =1$, $t \in \R$.
 We slightly generalize
the usual parabolicity assumption and show that still  $L^p$-estimates hold for  the
second spatial derivatives of $u$. 
  We also investigate   the dependence of
the constant appearing in such  estimates from the symmetric  $d \times d$-matrix $c(t)$ 
 $=\big(c_{ij}(t)\big)_{i,j =1, \ldots, d}$. In the final section we treat more general
equations involving   Ornstein-Uhlenbeck type operators and show that the previous
a-priori estimates are still true.
%In the final Section \ref{ou11} we will 
%show  an extension of such estimates to treat  Ornstein-Uhlenbeck %type
%operators. 

The $L^p$-estimates we are interested  in are  the following:
 for any $p \in (1, \infty)$, 
there exists $\tilde M >0$ such that, for any
$u \in C^{\infty}_0 (\R^{d+1})$ which solves \eqref{basic}, we have
% for $i,j =1, \dots, d,$ 
\begin{equation}
\label{ch2}
 \| u_{x_i x_j} \|_{L^p (\R^{d+1})} \le \tilde M 
 \, \| f \|_{L^p (\R^{d+1})}, \;\;\;  i,j =1, \dots, d,
\end{equation} 
 where the $L^p$-spaces are considered with respect to
the $d+1$-dimensional Lebesgue measure.
Usually, in the literature such a-priori estimates are stated
requiring that   there exists ${\lambda} $ and $\Lambda >0$
such that 
 \begin{equation}
 \lambda |\xi|^2 \le 
\sum_{i,j=1}^d c_{ij}(t) \xi_i \xi_j \le \Lambda |\xi|^2,\;\;\; t \in \R, \xi
\in \R^d,
\label{brut}
\end{equation}
where $|\xi|^2 = \sum_{i=1}^{d} \xi^2_i$. We refer to Chapter 4 in \cite{LSU}, Appendix in \cite{SV}, Section VII.3 in
\cite{L}, which  also assumes that $c_{ij}$ are uniformly continuous, and Chapter
4 in \cite{K3}. The proofs are  based on  parabolic extensions of the   Calderon-Zygmund  theory for
singular integrals (cf. \cite{FR} and \cite{J}). This theory was originally used    to prove a-priori Sobolev estimates for   the Laplace
equation (see \cite{CZ}).
% ; a first
%step in this direction is  a parabolic version of the
%Calderon-Zygmund decomposition in a cube.
%The Calderon-Zygmund theory ??
%theory also use  covering lemmas and 
%The seminal paper by  Calderon-Zygmund was originally used to prove
%a-priori $L^p$-estimates for second derivatives of solutions to the Laplace
%equations (see \cite{CZ}).
 In the above mentioned references, it is  stated that $\tilde M$ depends not only on $d$,
$p$, $\lambda$ (the parabolicity constant) but also on  $\Lambda$. An attempt to
determine the explicit dependence of $\tilde M$  from $\lambda $ and $\Lambda$ has been
done in  Theorem
A.2.4 of \cite{SV}  finding a quite complicate constant.

The fact that $\tilde M$ is actually \textit{independent} of $\Lambda$ 
is mentioned in Remark 2.5 of \cite{K2}. This property follows from a general
result given in  Theorem 2.2 of \cite{K1}. Once this independence from $\Lambda$
is proved one can use a rescaling argument (cf. Corollary \ref{main}) to show that 
we have
\begin{equation}
\label{depe}
\tilde M = \frac{M_0}{\lambda},
\end{equation} 
for a suitable positive constant $M_0$ depending only on $d$ and $p$.

 In Theorem \ref{kp} and Corollary \ref{main}   we generalize the parabolicity
condition by requiring that the symmetric $d \times d$ matrix
$c(t)  = \big( c_{ij}(t) \big)$ is non-negative definite,
for any $t \in
\R$, and, moreover, that
there exists and integer $p_0$, $1 \le p_0 \le d$,
and ${\lambda} \in (0, \infty) $
such that   
\begin{equation} \label{nondeg8}
%\langle c(t) \xi, \xi 
%\rangle =
 {\lambda}   \sum_{j=1}^{p_0}
\xi_j^2\ \le \, 
\sum_{i,j=1}^d c_{ij}(t) \xi_i \xi_j, \;\; t \in \R, \; \xi \in \R^d
\end{equation}
(cf. Hypothesis \ref{hy1} in Section \ref{sec:2}). We show that \eqref{nondeg8} is enough to get
estimates like \eqref{ch2} for $i, j = 1, \ldots, p_0$, with a constant $\tilde
M$ as in \eqref{depe} (now $M_0$ depends on $p, d$ and $p_0$). An example in which \eqref{nondeg8} holds is 
 \begin{equation} \label{eq2}
u_t(t,x,y) + u_{xx}(t,x,y) +
 t u_{xy} (t,x,y)+ t^2 u_{yy}(t,x,y) = f(t,x,y),
\end{equation}
 $(t,x,y) \in \R^3$  
(see Example \ref{exa}). In this case 
%$p_0 =1$, $\lambda = 3/4$
% and  
we have an a-priori estimates for $\|  u_{xx}\|_{L^p}$. 

We will first provide a  purely analytic proof of Theorem \ref{kp} in the case of $L^2$-estimates. This is based on Fourier transform techniques.
Then we provide the proof for the general case $1< p < \infty$ in Section 
\ref{p11}. 
 This proof is inspired  by the one of Theorem 2.2 in  
\cite{K1} and requires  the concept of stochastic integral with respect to the Wiener process.   In Section \ref{civuo}
we recall basic properties  of the stochastic integral.  
 It is not clear  how to prove Theorem \ref{kp} for $p \not = 2$ in a purely analytic way. One possibility could be  to follow step by step the proof given in Appendix of \cite{SV} 
%passing 
% through the  study of singular integrals operators and 
trying to improve the constants appearing in the various estimates. 
%However we are not able to get Theorem \ref{kp}  in this way even when $d= %p_0$.
% Indeed we can  improve some estimates in \cite{SV} but this is still not %enough to obtain the assertion of Theorem \ref{kp} even when $p_0 =d$.    

In Section \ref{ou11} we will extend our estimates to  more general equations
like 
\begin{equation}
\label{ou24}
u_t (t,x) + \sum_{i,j=1}^d c_{ij}(t) u_{x_i x_j}(t,x) +
 \sum_{i,j=1}^d a_{ij} x_j \, u_{x_i}(t,x) 
= f(t,x),
\end{equation}
where  $A= (a_{ij})$ is  a given real $d \times d$-matrix. 
% Assuming \eqref{nondeg8} with $p_0 =d$ we prove that
% for any $p \in (1, \infty)$, $T>0$, 
%there exists $M_0 = M_0 (d,p, T, A) >0$ such that, for any
%$u \in C^{\infty}_0 ((-T, T) \times \R^d)$ which solves \eqref{basic}, we %have 
%\eqref{ch2} with $\tilde M = \frac{M_0}{\lambda}$ (see Theorem \ref{ouma} for %a  more general statement). 
If \eqref{nondeg8} holds with $p_0 =d$  then we show that  estimate 
 \eqref{ch2} is still true with
% $\tilde M = {M_0}/{\lambda}$,
 $M_0 = M_0 (d,p, T, A) >0$  for any  solution $u \in C^{\infty}_0 ((-T, T) \times \R^d)$ of \eqref{ou24} (see   Theorem \ref{ouma} for a more general statement).
% for the  more general statement). 

An interesting case of \eqref{ou24} is  
when $c(t)$ is constant, i.e., $c(t) = Q$, $t \in \R$. Then  equation
 \eqref{ou24}   becomes
$$
u_t + {\cal A} u  =f,
$$
where  ${\cal A}$ is  the   Ornstein-Uhlenbeck operator, i.e.,
\begin{equation}
\label{aaa}
{\cal A} v(x) = \text{Tr}(QD^2 v(x)) + \langle Ax, D v(x) \rangle,
\;\;\; x \in \R^{d}, \;\; v \in  C_{0}^{\infty}(\R^{d}).
\end{equation} 
%we may consider t
%he 
%is the   Ornstein-Uhlenbeck operator  on $\R^d$.
 The operator ${\cal A}$ and its parabolic counterpart ${\cal L} = 
  {\cal A}  - \partial_t $, which is also called Kolmogorov-Fokker-Planck
 operator, have recently  received much attention (see, for instance,  \cite{BCLP1}, \cite{BCLP2}, \cite{DL}, \cite{DZ}, \cite{FL}, \cite{LP}, \cite{MPRS}, \cite{P} and the references therein).  
   The   operator ${\cal A}$   is the  generator 
of the Ornstein-Uhlenbeck  process which 
%semigroup associated to the   
%stochastic differential equation:
%$$
%dX_t= AX_tdt+\sqrt{2}\, Q^{1/2}dW_t,\;\;\;t>0, \;\; X_0=x \in \R^d, 
%$$
%where $(W_t)$ is a standard Wiener process taking values in
%$\mathbb{R}^{d}$. This 
solves a linear stochastic differential equation (SDE) 
%equation can 
describing the random motion of a
particle in a fluid (see \cite{OU}). 
 Several interpretations in physics and
finance for  $\mathcal{A}$ and  ${\cal L} $ 
 are explained in the survey \cite{Pa}. 
From the a-priori estimates for the parabolic equation \eqref{ou24}
one can deduce  elliptic estimates  like
\begin{equation}
\label{esti114}
\|  v_{x_i x_j} \|_{L^p(\R^{d})} \le \, C_1\,  \big(
 \| {\cal A} v\|_{L^p(\R^{d})} + \| v\|_{L^p(\R^d)} \big),
\end{equation} 
with $C_1 = \frac{M_2(d,p, A)} {\lambda} $,  assuming that ${\cal A}$ is non-degenerate
(i.e., $Q$ is positive definite; see Corollary
 \ref{ouu}). Similar  estimates have been already obtained in \cite{MPRS}. 
 Here  we  can show in addition the precise dependence of the constant $C_1$ from the matrix $Q$.
 
More generally,  estimates  like   \eqref{esti114} 
hold for possibly degenerate hypoelliptic Ornstein-Uhlenbeck operators ${\cal A}$ (see \cite{BCLP1});
 a typical example in $\R^2$ is   
  ${\cal A} v =$ $ q v_{xx} + x v_{y} $  with $q>0$ (cf.  Example \ref{ko}).
In this case we have 
\begin{equation}
\label{like}
\|  v_{x x} \|_{L^p(\R^{2})} \le \, C_1  \big(
 \|  q v_{xx} + x v_{y}\|_{L^p(\R^{2})} + \| v\|_{L^p(\R^2)} \big).
\end{equation}
 Estimates as \eqref{like}  
have been deduced in \cite{BCLP1} by corresponding parabolic estimates for 
$ {\cal A} - \partial_t$, using that such operator  
 is left invariant with respect to a suitable Lie group structure on $\R^{d+1}$ (see \cite{LP}). We also mention  
 \cite{BCLP2} which contains  a generalization of \cite{BCLP1} 
%this result to possibly degenerate operators as in \eqref{aaa} 
when $Q $  may also depend on $x$ and \cite{P} where 
 the results in \cite{BCLP1} are used 
%are also useful 
to study well-posedness of  related SDEs.
% (see \cite{P}).   
 Finally, we point out   that   in the degenerate hypoelliptic case considered in \cite{BCLP1} it is  not  clear how to prove  the precise dependence of the constant appearing in the a-priori $L^p$-estimates from the matrix $Q$.

\bigskip 
%\noindent \textbf{Notation.}
%\paragraph {N} 
We denote by $|\cdot | $ the usual euclidean norm in any $\R^k$, 
 $k \ge 1$. Moreover,
 $\langle \cdot , \cdot \rangle$ indicates the usual inner product in $\R^k$. 

We  denote by $L^p (\R^k)$, $k \ge 1$,  $1< p < \infty$ the usual Banach 
spaces of measurable real functions $f$ such that $|f|^p$ is integrable on
$\R^k$ with respect to the Lebesgue measure. The space of all  
$L^p$-functions $f : \R^k \to \R^j$ with $j > 1$ is indicated with $L^p (\R^k ; \R^j)$.
 Let $H$ be an open set in $\R^k$; $C^{\infty}_0 (H)$ stands for the vector
space of all real $C^{\infty}$-functions $f: H \to \R$ which have  compact
support. 

Let $d \ge 1$. Given a regular function  $u: \R^{d+1} \to \R$, we denote by 
 $D^2_x u (t,x)$ the $d \times d$ Hessian matrix of $u$ with respect to  the spatial variables  at $(t,x) \in \R^{d+1}$,
i.e.,  $D^2_x u (t,x) $ $= (u_{x_i x_j 
} (t,x))_{i,j =1, \ldots, d}$.
%  where $u_{x_i x_j}$ are standard partial derivatives
%of $u$. 
 Similarly we define the gradient $D_x u(t,x) \in \R^d$
 with respect to the spatial variables. 

 Given a real $k \times k $ matrix $A$, $\| A \|$ denotes its operator norm
and $Tr(A)$ its trace. 

\smallskip
Let us recall the notion of {\it Gaussian measure} (see, for instance, Section
1.2 in \cite{Bo} or Chapter 1 in \cite{DZ} for more details). Let $d \ge 1$. 
Given a symmetric non-negative definite  $d \times d$ matrix  $Q$, the symmetric
Gaussian
measure $N(0,Q)$ is  the unique Borel probability measure on $\R^d$ such that
its Fourier transform is  
\begin{equation}
\label{gau}
 \int_{\R^d} e^{i \langle x, \xi \rangle} \; N(0, Q) (dx)
= e^{- \langle \xi, Q \xi\rangle},\;\;\; \xi \in \R^d;
\end{equation}
$N(0,Q)$ is  the Gaussian measure with mean 0 and covariance
matrix  $2 Q$. 
 If in addition $Q $ is positive definite than $N(0,Q)$ has the following density $f$ with
respect to the $d$-dimensional Lebesgue measure 
\begin{equation}
\label{den}
f(x) = \frac{1}{ \sqrt {(4 \pi)^n \, \text{det}(Q)}}e^{- \frac{1}{4}\langle
Q^{-1} x, x\rangle},\;\; x \in \R^d.
\end{equation} 
Given two Borel probability measures $\mu_1$ and $\mu_2$ on $\R^d$ 
the convolution $\mu_1 * \mu_2$ is the Borel probability measure  defined as 
$$
 \mu_1 * \mu_2 (B) = \int_{\R^d} \int_{\R^d} 1_{B} (x+ y) \mu_1 (dx) \mu_2 (dy) = \int_{\R^d} \mu_1(dx) \int_{\R^d} 1_{B} (x+ y)  \mu_2 (dy),
$$
for any Borel set $B \subset \R^d$. Here $1_{B}$ is the indicator function of
$B$ (i.e., $1_B(x) = 1 $ if $x \in B$ and $1_B(x) = 0 $ if $x \not \in B$).
 It can be easily verified that 
%if $N(0,Q) $ and $N(0, R)$ are two Gaussian 
%measures on $\R^d$ then
\begin{equation}
\label{conv1}
 N(0,Q) * N(0, R) = N(0, Q+ R),
\end{equation} 
where $Q+R$ is the sum of the two symmetric  non-negative definite matrices $Q$ and $R$.

%\vskip 2cm
%usa ${\lambda}$ anziche' ${\lambda} $

\section{A-priori $L^p$-estimates 
%on $\R^{d+1}$ 
}
\label{sec:2}

In this section we consider parabolic equations like \eqref{basic}.
%\begin{equation*}
% u_t(t,x) + \sum_{i,j=1}^d c_{ij}(t) u_{x_i x_j}(t,x) 
% = f(t,x),\;\;\;
% (t,x) \in \R^{d+1}.
% \label{basic}
%\end{equation*} 

We always assume that the coefficients \textsl{$c_{ij}(t)$ of the symmetric
$d \times d $ matrix $c(t)$ appearing in \eqref{basic}
are (Borel) measurable and locally bounded on $\R$ and, moreover,  that $\langle
c(t) \xi, \xi  \rangle \ge 0$, $t \in \R$, $\xi \in \R^d$}. 
Moreover, we will consider the symmetric non-negative 
 $d \times d$ matrix
\begin{equation}
 C_{sr} = \int_s^r c(t)dt,\;\;\; s \le r,\;\; s,r \in \R.
 \label{crs}
\end{equation} 
We start with a simple representation formula for solutions to equation
\eqref{basic}. This formula is usually obtained assuming that $c(t)$ is
uniformly positive. However there are no difficulties to 
prove it even in the present case when $c(t)$ is only non-negative definite.
\begin{proposition} \label{rap}
Let $u \in C^{\infty}_0(\R^{d+1})$ be a solution to \eqref{basic}. Then we have,
 for $(s,x) \in \R^{d+1}$, 
\begin{equation}\label{rap1}
 u(s,x) = -\int_s^{\infty} dr \int_{\R^d} f(r, x +y ) N(0,{C_{s r}}) (dy).
\end{equation}
\end{proposition}
\begin{proof} Let us denote by $\hat u (t, \cdot)$ 
 the Fourier transform of $u(t, \cdot)$ in the space variable $x$. Applying such partial Fourier transform
to  both sides of \eqref{basic} we obtain
$$
\hat u_t(s,\xi)  -  \sum_{i,j =1}^d c_{ij}(s) \xi_i \xi_j  \hat u(s, \xi) 
=
\hat f(s,\xi),
$$
i.e., we have
\begin{equation} \label{krr}
\hat u (s, \xi) = - \int_s^{ \infty} e^{- \langle C_{sr} \xi, \xi \rangle
}  \hat f (r, \xi) dr ,\;\;\; (s, \xi ) \in \R^{d+1}.
\end{equation}
It follows that 
$$
\hat u (s, \xi) = - \int_s^{ \infty} \Big( \int_{\R^d} e^{i \langle x, \xi
\rangle} N(0, C_{sr}) (dx) \Big)  \hat f (r, \xi) dr. 
$$
By some straightforward computations, using also the uniqueness property of the
Fourier transform, we  get 
\eqref{rap1}.

Alternatively, starting from \eqref{krr} one can directly follow  the
computations of pages 48 in \cite{K3} and  obtain \eqref{rap1}.  
 These computations use that there exists 
${\epsilon} >0$ such that $ \langle c(t) \xi , \xi \rangle \ge {\epsilon}
|\xi|^2$, $\xi \in \R^d$.  
We write, for $\epsilon >0$, using the Laplace operator,
$$
 u_t(t,x) + \sum_{i,j=1}^d c_{ij}(t) u_{x_i x_j}(t,x) 
 + \epsilon \triangle u(t,x)
 = f(t,x) + \epsilon \triangle u(t,x),
 $$
 $(t,x) \in \R^{d+1};$ since
 $c(t ) + \epsilon I$ is uniformly  positive, following 
\cite{K3} we find 
\begin{gather*}
u(s,x) = -\int_s^{\infty} dr \int_{\R^d} f(r, y+x) N(0,{C_{s r}} + \epsilon
(r-s)I) (dy) 
\\
-  \epsilon \int_s^{\infty} dr \int_{\R^d} 
\triangle u(r, y+x) N(0,{C_{s r}} + \epsilon (r-s)I) (dy).
\end{gather*}
Using also \eqref{conv1} we get 
\begin{gather*}
u(s,x)
= -\int_s^{\infty} dr \int_{\R^d} N(0,(r-s)I)(dz)\int_{\R^d} f(r, x + y +
\sqrt{\epsilon}\, z) N(0,{C_{s r}}) (dy) 
\\
- \epsilon \int_s^{\infty} dr \int_{\R^d} N(0,(r-s)I)(dz)\int_{\R^d} \triangle u
(r, x + y + \sqrt{\epsilon}\, z) N(0,{C_{s r}}) (dy). 
\end{gather*}
Now we can pass to the limit as $\epsilon \to 0^+$
by the Lebesgue theorem and get \eqref{rap1}.
 
\end{proof}

%STOP 
%
%Lemma
%Magari metti $\lambda >0$ 
 
The next assumption is a slight generalization of the usual parabolicity
 condition which corresponds to the case $p_0 =d$ (see also Remark
\ref{general}). 
\begin{hypothesis} \label{hy1}
The coefficients $c_{ij}$ are  locally bounded on $\R$ and
 the matrix $c(t) = \big( c_{ij}(t) \big)$ is symmetric  non-negative
definite, $t \in \R$.
 In addition, there exists an integer $p_0$, $1 \le p_0 \le d$, and 
${\lambda}  \in (0, \infty)$
such that   
\begin{equation} \label{nondeg}
\langle c(t) \xi, \xi 
\rangle =\sum_{i,j=1}^d c_{ij}(t) \xi_i \xi_j \ge {\lambda}   \sum_{j=1}^{p_0}
\xi_j^2,\;\;\; t \in \R,\; \xi \in \R^d.
\end{equation} 
\end{hypothesis}
A possible generalization of this hypothesis is given in Remark \ref{general}.  Note that if we introduce  the orthogonal projection 
\begin{equation}
\label{on4}
I_0 : \R^d \to F_{p_0},
\end{equation} 
where $F_{p_0}$ is  the subspace
generated by $\{e_1, \ldots, e_{p_0} \}$  (here $\{ e_i\}_{i=1, \ldots, d}$ denotes the canonical basis in $\R^d$) then \eqref{nondeg7}
can be rewritten as
\begin{equation}
\label{nondeg7}
\langle c(t) \xi, \xi 
\rangle  \ge {\lambda}    |I_0 \xi
|^2,\;\;\; t \in \R, \; \xi \in \R^d.
\end{equation} 
\begin{lemma} \label{ser} Let $g : \R^{d+1} \to \R$ be Borel, bounded, with compact
support
 and such
that 
$g(t, \cdot) \in C_0^{\infty}(\R^{d})$, $t \in \R$. 
Fix $i, j \in \{
1, \ldots,p_0 \}$ and consider
 \begin{equation*}
 w_{ij}(s,x) = -\int_s^{\infty} dr \int_{\R^d} g_{x_i x_j}(r, x+y)
N(0,  I_0 (r-s)) (dy), \; (s,x) \in \R^{d+1},
\end{equation*}
where $I_0$ is defined in \eqref{on4}.
 For any $p \in (1, \infty)$,   there
exists $M_0 = M_0 (d,p, p_0)$ $ >0$,  such that 
\begin{equation}
 \|  w_{ij} \|_{L^p(\R^{d+1})} \le M_0 \| g\|_{L^p(\R^{d+1})}.
\label{ne}
\end{equation} 
\end{lemma}
\begin{proof} If $p_0 =d$ the estimate is classical.
% and can be obtained by the
%theory of singular integrals. 
In such case we are dealing with the  heat
equation 
$$
\partial_t u +  \triangle u = g 
$$
on $\R^{d+1}$ and $w_{ij}$ coincides with the second partial derivative with respect to
${x_i}$ and $x_j$ of the heat potential applied to $g$ 
(see, for instance, page 288 in \cite{LSU} or Appendix in \cite{SV}). 
%?? krylov priola ??
% Under our assumption one can even prove that ??? 
% $$
% u(s,x) = \int_s^{\infty} dr \int_{\R^d} f(r, x+y)
%N(0,  I_0 (r-s)) (dy), \; (s,x) \in \R^{d+1}
% $$
%is a solution to \eqref{cal} (see Lemma 3.3 in \cite{KP}).
If $p_0 < d$ we write $x = (x', x'' )$ for $x \in \R^{d}$, where $x' \in
\R^{p_0}$ and $x'' \in \R^{d-p_0}$. We get 
\begin{gather*}
w_{ij}(s,x', x'') = - \int_s^{\infty} dr 
\int_{\R^{p_0}} g_{x_i x_j}(r, x'+ \, y', x'')
N(0,  I_{p_0} (r-s)) (dy'), 
\end{gather*}
where $I_{p_0}$ is the identity matrix in $\R^{p_0}$. Let us fix $x'' \in 
 \R^{d-p_0}$
and consider the function $l(t, x') = g(t , x' , x'')$
 defined on $\R \times \R^{p_0}$. By classical estimates for the heat
equation  $\partial_t u $ $+  \triangle u$ $ = l$
 on $\R^{p_0 +1}$ we obtain
$$ 
 \int_{\R^{p_0 +1}} |w_{ij} (s,x', x'')|^p ds dx' \le 
% M_0^p 
%   \int_{\R^{p_0 +1}} |l (s,x')|^p ds dx'
%=
M_0^p 
   \int_{\R^{p_0 +1}} |g (s,x', x'')|^p ds dx'. 
$$
Integrating with respect to $x''$ we get the assertion.
 
\end{proof}

In the sequel we also consider the differential operator $L$
\begin{equation} \label{lu}
 L u(t,x) = \sum_{i,j=1}^d c_{ij}(t) u_{x_i x_j}(t,x),
\;\; (t,x) \in \R^{d+1}, \;\; u \in  C_{0}^{\infty}(\R^{d+1}).
\end{equation}
The next  regularity result when $p_0 =d$
follows by a general result given in Theorem 2.2 of \cite{K1} (cf. Remark
2.5 in \cite{K2}).

In the next two sections we  provide the proof. First we give  a 
direct and self-contained proof in the case $p=2$ by  Fourier transform
 tecniques
(see Section \ref{p2}). 
Then in Section \ref{p11}  we consider the general case.  
The 
proof for $1<p <\infty$ 
%of the general case is independent from the case $p=2$ 
% and 
is inspired by the one of Theorem 2.2 in  
\cite{K1} and  uses also a probabilistic argument. This argument is used to ``decompose'' a suitable
Gaussian measure in order  to apply   successfully the  Fubini theorem 
(cf. \eqref{crucial} and \eqref{frt5}).

 We stress again  that in the  case of 
  $d = p_0$, usually,  
the next result is stated under the stronger assumption that  \eqref{nondeg}
holds with ${\lambda}  =1$ and also  that $c_{ij}$ are {\it bounded,} i.e.,  
 assuming \eqref{brut} with $\lambda =1$ and   $\Lambda \ge 1 $
% it is also claimed that 
% the constant $M_0$  
%depends  on
%$\Lambda$ 
(see, for instance, Appendix in  \cite{SV} and \cite{LSU}).

\begin{theorem}\label{kp}  Assume Hypothesis \ref{hy1}
 with ${\lambda}  =1$ in \eqref{nondeg}. Then, for $p \in (1, \infty)$,
there exists  a constant
 $M_0= M_0(d,p, p_0)$ such that, for any $u \in C_0^{\infty}(\R^{d+1})$, 
 $i,j =1, \ldots, p_0$, we have
 \begin{equation}
  \|  u_{x_i x_j} \|_{L^p(\R^{d+1})} \le M_0 \| u_t + Lu\|_{L^p(\R^{d+1})}. 
  \label{new}
 \end{equation} 
\end{theorem}

As a consequence of the previous result we obtain
\begin{corollary} \label{main} Assume Hypothesis \ref{hy1}.
 Then,
for any $u \in C_0^{\infty}(\R^{d+1})$, $p \in (1, \infty)$, 
 $i,j =1, \ldots, p_0$, we have (see \eqref{lu})
 \begin{equation}
  \|  u_{x_i x_j} \|_{L^p(\R^{d+1})} \le \frac{M_0}{{\lambda} } \; \| u_t +
Lu\|_{L^p(\R^{d+1})}. 
  \label{new1}
  \end{equation}
  where $M_0= M_0 (d, p, p_0)$ is the same constant appearing in \eqref{new}.
\end{corollary}
\begin{proof}  Let us define $w(t,y) = u(t, {\sqrt{{\lambda} }}\, y )$. Set $f =
u_t + Lu$; since
$u(t,x) = w(t, \frac{x}{\sqrt{{\lambda} }})$, we find 
%\begin{gather*}
\begin{equation*}
f (t, {\sqrt{{\lambda} }} \, y) = w_t(t,y) +  \frac{1}{{\lambda} } \, Lw (t,y)
\end{equation*}
Now the matrix $(\frac{1}{{\lambda} } c_{ij})$ 
satisfies $\frac{1}{{\lambda} }\sum_{i,j=1}^d c_{ij}(t) \xi_i
\xi_j \ge  \sum_{j=1}^{p_0} \xi_j^2,$ $ t \in \R,$ $ \xi \in \R^d.$ Applying
Theorem \ref{kp} to $w$ we find 
\begin{equation*}
 \| w_{x_i, x_j}  \|_{L^p} \le {M_0}{{\lambda} ^{- \frac{d}{2p}}}  \| f \|_{L^p}
\end{equation*}
and so 
\begin{equation*}
 {\lambda} ^{1 - \frac{d}{2p}} \| u_{x_i, x_j}  \|_{L^p} \le M_0 
{\lambda^{-\frac{d}{2p}}}  \| f
\|_{L^p}
\end{equation*}
which is the assertion.
 
\end{proof}

\begin{example} \label{exa} {\em The equation  \eqref{eq2}
%\begin{equation} \label{eq2}
%u_t(t,x,y) + u_{xx}(t,x,y) +
% t u_{xy} + t^2 u_{yy} = f(t,x,y),
%\end{equation}
% $(t,x,y) \in \R^3$, 
verifies the assumptions of Corollary \ref{main} with $p_0=1$ and $\lambda = 3/4$ since
 $$
 \sum_{i,j=1}^2 c_{ij}(t) \xi_i \xi_j  
 = \xi_1^2 + t \xi_1 \xi_2 + t^2 \xi^2_2 
 \ge \frac{3}{4} \xi_1^2, \;\;
 (t,\xi_1, \xi_2) \in \R^{3}.
$$
 Hence  there exists $M_0 >0$ 
 such that if $u \in
C^{\infty}_0(\R^3)$ solves \eqref{eq2} then
$$
  \|  u_{x x} \|_{L^p(\R^{3})} \le \frac{M_0}{{\lambda} } 
  \; \| f \|_{L^p(\R^{3})}. 
$$ }
\end{example}

\begin{remark} \label{general} {\em One can easily generalize  Hypothesis
\ref{hy1} as follows:

 the coefficients $c_{ij}$ are  locally bounded on $\R$ and, moreover, there
exists an orthogonal projection $I_0 : \R^d \to \R^d $ and  
${\lambda} >0$
such that, for any $t \in \R$, a.e.,   
\begin{equation} \label{nondeg2}
\langle c(t) \xi, \xi 
\rangle 
%=\sum_{i,j=1}^d c_{ij}(t) \xi_i \xi_j \
\ge {\lambda}  |I_0\xi|^2,\;\; \xi \in \R^d.
\end{equation} 
Theorem \ref{kp} and Corollary \ref{main} continue to hold under this 
assumption. 

Indeed following the proof of Theorem \ref{kp} it is clear that  assertion \eqref{new}  can be obtained if assumption \eqref{nondeg} is satisfied  only for $t \not \in  B$ where $B \subset \R$ is a Borel set of Lebesgue measure zero.  
 Moreover, if  \eqref{nondeg2} holds then  by a suitable linear change of variables in equation
\eqref{basic} we may assume that 
$I_0 (\R^d)$ is the linear subspace generated by $ \{ e_1, \ldots, e_{p_0} \}$ 
 for some $p_0$ with $1 \le p
\le d$ and so    apply Theorem \ref{kp}.
 
 Under hypothesis \eqref{nondeg2}  assertion \eqref{new} in Theorem
\ref{kp} becomes
$$
 \|  \langle D_x^2 u (\cdot)  h, k \rangle \|_{L^p(\R^{d+1})} \le M_0 \| u_t +
Lu\|_{L^p(\R^{d+1})},
$$
where $h, k \in I_0 (\R^d)$.}
% and $D^2_x u $ denotes the $d \times d $ Hessian
%matrix of $u$ in the spatial variables.
\end{remark}

\subsection{Proof of Theorem \ref{kp} when $p=2$}
\label{p2}
%CASO $\lambda$ ?? 

%The proof uses the Fourier transform  and is inspired by the one of the %related Lemma A.2.2
%in \cite{SV}.
% Note that  Lemma A.2.2 has $p_0=d$;   moreover,  the additional condition  
%  \eqref{brut} is assumed in \cite{SV} and 
% the final constant $M_0$ appearing in 
%\eqref{new} also depends on
%$\Lambda$. 
% In the sequel we will improve some estimates in \cite{SV}.
This  proof  is inspired by the one of  Lemma A.2.2
in \cite{SV}.
 Note that  such lemma  has $p_0=d$ and, moreover,  it assumes  the stronger condition  
  \eqref{brut}. In Lemma A.2.2  the  constant $M_0$ appearing in 
\eqref{new} is  $ 2 \sqrt{\Lambda}$.

 We start from \eqref{krr} with
$$
f =  u_t + Lu.
$$
Recall that  for $g : \R^{d+1} \to \R$, 
$\hat g (s, \xi)$ denotes the Fourier transform of $g(s, \cdot) $ with respect to the $x$-variable
($s \in \R$, $\xi \in \R^d$) assuming that $g(s, \cdot) \in L^1(\R^d)$. 

Let us fix $s \in \R$.
Let $i,j = 1, \ldots, p_0$. We easily compute the Fourier transform of $u_{x_i
x_j}(s, \cdot)$ (the matrix $C_{sr}$ is defined in \eqref{crs}):
\begin{gather*} 
\hat u_{x_i x_j} (s, \xi) =  - \xi_i \xi_j \, \hat u (s, \xi) 
=  \xi_i \xi_j   \int_s^{ \infty} e^{- \langle C_{sr}\,  \xi, \xi \rangle
}  \hat f (r, \xi) dr ,\;\;\;  \xi  \in \R^{d}.
\end{gather*}
Since $|I_0 \xi|^2 = \sum_{i=1}^{p_0} |\xi_i|^2 $, we  get 
$$
2 |\hat u_{x_i x_j} (s, \xi)| \le |I_0 \xi|^2
  \int_s^{ \infty} e^{- \big( \langle C_{0r} \, \xi, \xi \rangle - 
 \langle C_{0s} \,  \xi, \xi \rangle \big)
}  |\hat f (r, \xi)| dr =  G_{\xi}(s). 
$$ 
Now we fix $\xi \in \R^d$, such that $|I_0\xi| \not = 0$, and define  
$$
g_{\xi} (r)
 = \langle C_{0r} \, \xi, \xi \rangle 
  = \int_0^r \langle c(p) \xi , \xi \rangle dp, \;\; r \in \R.
$$
 Changing variable $t = g_{\xi} (r),$ we get 
$$
G_{\xi}(s) = |I_0 \xi|^2
  \int_{g_{\xi} (s)}^{ \infty} e^{ ( g_{\xi} (s) \,  - \, t)
}  \,  |\hat f ( g_{\xi}^{-1} (t) , \xi)| \frac{1}{\langle c( g_{\xi}^{-1} (t))
\xi, \xi \rangle } dt .
$$ 
 Let us introduce $\varphi (t) = e^{t} \cdot 1_{(- \infty, 0)}(t) $, $t \in \R$,
 and 
$$
F_{\xi}(t) =   |I_0 \xi|^2 \, |\hat f ( g_{\xi}^{-1} (t) , \xi)|
\frac{1}{\langle c( g_{\xi}^{-1} (t)) \xi, \xi \rangle }.
$$
Using the standard convolution for real functions defined on $\R$ we find
$$
G_\xi (s) = (\varphi * F_\xi ) \, (g_{\xi} (s)). 
$$
  Therefore (recall \eqref{nondeg7} with ${\lambda}  =1$)
\begin{equation}
\label{l2}
\| G_\xi \|_{L^2(\R)}^2 = \int_{\R}
 | (\varphi * F_\xi)(t)|^2 \frac{1}{\langle c( g_{\xi}^{-1} (t)) \xi, \xi
\rangle } dt 
%\end{equation}
%$$
\le \frac{1}{ |I_0 \xi|^2 }\| \varphi * F_\xi \|_{L^2 (\R)}^2.
\end{equation}
%$$
which implies $\| G_\xi \|_{L^2(\R)}  
\le \frac{1}{   |I_0 \xi| }\| \varphi * F_\xi \|_{L^2 (\R)}.
$ On the other hand,
using the Young inequality, we find, for any $\xi \in \R^d$ with $|I_0 \xi| \not = 0$,
\begin{gather*}
\| \varphi * F_\xi  \|_{L^2 (\R)} \le  \| \varphi \|_{L^1(\R)} 
\, \|   F_\xi  \|_{L^2 (\R)} = \|   F_\xi  \|_{L^2 (\R)}
\\
 = |I_0 \xi|^2  \Big( \int_{\R}
|\hat f ( g_{\xi}^{-1} (t) , \xi)|^2 
\frac{1}{(\langle c( g_{\xi}^{-1} (t)) \xi, \xi \rangle)^2 } dt \Big)^{1/2} 
\\
=  |I_0 \xi|^2  \Big( \int_{\R}
|\hat f ( r , \xi)|^2 
\frac{1}{(\langle c( r) \xi, \xi \rangle)^2 } \langle c( r) \xi, \xi \rangle 
dr \Big)^{1/2} 
\\
\le \frac{|I_0 \xi|^2}{  |I_0 \xi| }  \Big( \int_{\R}
|\hat f ( r , \xi)|^2 
  \Big)^{1/2}  = {|I_0 \xi|} \cdot 
\| \hat f ( \cdot , \xi ) \|_{L^2(\R)}.
\end{gather*}
Using also \eqref{l2} we obtain, for any $\xi \in \R^d$, $|I_0 \xi| \not =0$,
\begin{gather*}
2 \| |\hat u_{x_i x_j} (\cdot, \xi) \|_{L^2(\R)} 
 \le \| G_\xi \|_{L^2(\R)}  \le \, 
 \| \hat f ( \cdot , \xi ) \|_{L^2(\R)}.
\end{gather*}
From the  previous inequality, integrating  with respect to $\xi$ over $\R^d$ we find
$$
4 \int_{\R} ds \int_{\R^d} |\hat u_{x_i x_j} (s, \xi)|^2 d \xi
 \le \,   \int_{\R} ds \int_{\R^d} |\hat f (s, \xi)|^2 d \xi.  
$$
By using the  Plancherel  theorem in $L^2(\R^d)$ we easily obtain
\eqref{new} 
 with $M_0 =1/2$. The proof is complete.  

\subsection{Proof of Theorem \ref{kp} when $1 < p < \infty$}
\label{p11}

The proof uses the concept of stochastic integral in a crucial point (see
\eqref{crucial} and \eqref{frt5}). Before starting the proof  we collect  some  basic properties
of the stochastic integral with respect to the Wiener process which are needed
(see, for instance, Chapter 4 in \cite{A} or Section 4.3 in \cite{SV} for more details).

\subsubsection {The stochastic integral}
\label{civuo}  
Let $W = (W_t)_{t \ge 0}$ be a standard $d$-dimensional Wiener process
defined on a probability space $(\Omega, {\cal F}, \P) $. Denote by $\E$ the
expectation with respect to $\P$.  

Consider a function $F \in L^2 ( [a,b];
\R^{d} \otimes \R^d)$ (here $0 \le a\le b $ and $\R^{d} \otimes \R^d$ denotes
the space of all real $d \times d $-matrices). 

Let $(\pi_n)$ be any sequence of partitions of $[a,b]$ such that 
$|\pi_n| \to 0$ as $n \to \infty$ (given a partition  $\pi = \{ t_0 =a$, ...
$,t_N =b \}$ we set
 $|\pi| = \sup_{t_k  , t_{k+1} \in \pi  }  |t_{k+1} - t_k|$).
 One defines the
stochastic integral $\int_a^b F(s) dW_s$
as the  limit in $L^2 (\Omega, \P; \R^d)$ of 
$$
J_n = \sum_{t_k^n,\, t_{k+1}^n \in \pi_n} F(t_{k}^n) \, 
(W_{t_{k+1}^n} - W_{t_{k}^n} ).
$$
as $n \to \infty$ (recall that the previous formula  means
 $$J_n (\omega) = \sum_{t_k^n,\, t_{k+1}^n \in \pi_n} F(t_{k}^n) \, 
(W_{t_{k+1}^n}(\omega)  - W_{t_{k}^n}(\omega) ),
$$ for any $\omega \in
\Omega$). 
 One can prove that the previous limit is independent of the choice of
$(\pi_n)$. Moreover, we have, $\P$-a.s.,  
\begin{equation} \label{stoc}
\int_a^b F(s) dW_s = \int_0^b F(s) dW_s - 
\int_0^a F(s) dW_s. 
\end{equation} 
 Set $\Gamma_{ab} = \int_a^b F(s)F^*(s) ds$ where $F^*(s)$ denotes the adjoint
matrix of $F(s)$. Clearly, $\Gamma_{ab}$ is a $d \times d$ symmetric 
non-negative definite matrix.  
 Moreover, we have (see, for instance, page 77 in \cite{A})
\begin{equation} \label{equ}
\E \big[ e^{i \, \sqrt{2}\langle \int_a^b F(s) dW_s, \xi \rangle} \big]
 =  \int_{\Omega} e^{i \, \sqrt{2}\, \langle \big(\int_a^b F(s) dW_s\big)
(\omega)\, , \, \xi \rangle}\,\, \P (d \omega) 
\end{equation}
$$
= \int_{\R^d} e^{i \langle x, \xi \rangle} \; N(0, \Gamma_{ab}) (dx)
= e^{- \langle \xi,  \Gamma_{ab} \xi\rangle},\;\;\; \xi \in \R^d. 
$$
Formula \eqref{equ} is equivalent to require that 
  for any Borel and bounded $f : \R^d \to \R$, 
\begin{equation} \label{se3}
 \E \Big[f \Big( \sqrt{2} \int_a^b F(s) dW_s \Big) \Big] = \int_{\R^d} f( y)
N(0,{\Gamma_{ab}}) (dy). 
\end{equation} 
Equivalently, one can say that  the law (or image measure) of $
\sqrt{2}\int_a^b F(s) dW_s$ 
is $N(0,\Gamma_{ab}) $.

\subsubsection{Proof of the theorem}

It is convenient to suppose that
 $u(t, \cdot) =0$ if $t \le 0 $ so that
$u \in C^{\infty}_0([ 0, \infty)\times \R^{d})$.

Indeed 
 if  $u(t, \cdot) =0$, $t \le T $, for some $T \in \R$, then  we can introduce 
 $v(t,x) = u(t+ T,x)$ which belongs to  $u \in C^{\infty}_0( [ 0, \infty)\times
\R^{d})$; from the a-priori estimate for $v_{x_i x_j}$ it follows \eqref{new} since $
 \| v_{x_i x_j} \|_{L^p(\R^{d+1})} $ $ =   \| u_{x_i x_j} \|_{L^p(\R^{d+1})}.
$

 We know that, for $s \ge 0$.  $x \in \R^{d}$,
\begin{equation*}
 u(s,x) = -\int_s^{\infty} dr \int_{\R^d} f(r, x+ y) N(0,{C_{sr}}) (dy),
\end{equation*}
where $f = u_t + Lu $ 
is bounded, with compact support on $\R^{d+1}$ and such
that 
$f(t, \cdot) \in C_0^{\infty}(\R^{d})$, $t \ge 0$. 
Let us fix $i, j \in \{
1, \ldots,p_0 \}$.

Differentiating under the integral sign  it is
not difficult to prove that 
\begin{equation*}
 u_{x_i x_j}(s,x) =- \int_s^{\infty} dr \int_{\R^d} f_{x_i x_j}(r, x+y)
N(0,{C_{sr}}) (dy).
\end{equation*}
Let us fix $s$ and $r$, $0 \le s \le r$, and consider 
\begin{equation*}
C_{sr} =   A_{sr} + (r-s)I_0,\;\;\; \text{where} \; 
 A_{sr} = \int_{s}^{r} (c(t) - I_0)dt.
\end{equation*}
By \eqref{conv1} 
we know that $N(0, C_{sr}) = N(0, A_{sr}) * N(0, (r-s)I_0) $ and so
\begin{equation} \label{ci1}
  \int_{\R^d} f_{x_i x_j}(r, x+y)
N(0,{C_{sr}}) (dy)
\end{equation}
$$
 = \int_{\R^d} N(0,{A_{sr}}) (dz)
\int_{\R^d} f_{x_i x_j}(r, x+ y + z) N(0, (r-s)I_0)(dy).
$$
Now we introduce a standard $d$-dimensional Wiener process $W = (W_t)_{t \ge 0}$
on a  probability space $(\Omega, {\mathcal F}. \P)$ (see Section \ref{civuo}).
Consider 
the symmetric $d \times d$ square root $\sqrt{c(t)-
I_0}$ of $c(t)-I_0$ and define the stochastic integral
$$
\Lambda_{sr}= \sqrt{2} \int_s^r \sqrt{c(t) -I_0} \; dW_t.
$$
By \eqref{stoc} we know that 
$$ \Lambda_{sr}  = b_r  - b_s, \;\; \text{where} 
\; b_t = \sqrt{2}\int_0^t  \sqrt{c(p) -I_0} \; dW_p,
$$ 
 $t \ge 0$, and 
 $b_t =0$ if $t \le 0$.  Moreover (cf. \eqref{se3}) 
 for any Borel and bounded $g : \R^d \to \R$, we have
\begin{equation} \label{crucial}
 \E[g(b_r - b_s)] = \int_{\Omega} g\big( b_r(\omega) \,  - \,  b_s (\omega)\big) \P (d \omega)
 =
 \int_{\R^d} g( y) N(0,{A_{sr}}) (dy). 
\end{equation}  
 Using this fact and  the
 Fubini theorem we get from \eqref{ci1}
\begin{gather} 
\nonumber  \int_{\R^d} f_{x_i x_j}(r, x +y)
N(0,{C_{sr}}) (dy)
\\
\nonumber = \E \Big [ \int_{\R^d} 
 f_{x_i x_j}(r, x+ y+ \Lambda_{rs}) \, N(0, (r-s)I_0)(dy) \Big]
\\
%\begin{equation} \label{frt5}
= \E \Big [ \int_{\R^d} 
 f_{x_i x_j}(r, x+ y + b_r - b_s) \, N(0, (r-s)I_0)(dy)
 \Big].
 \label{frt5}
 \end{gather}
Therefore we find
\begin{equation}
u_{x_i x_j}(s,x) = -\E \Big [ \int_s^{\infty}dr \int_{\R^d} 
 f_{x_i x_j}(r, x+ y + b_r - b_s) N(0, (r-s)I_0)(dy)
 \Big].
\end{equation} 
Now we estimate the $L^p$-norm of $u_{x_i x_j}$. To simplify the notation
in the sequel
we set 
 $N\big(0, (r-s)I_0\big) = \mu_{sr}$.
 Using the Jensen
inequality and the Fubini theorem we get 
\begin{gather*}
 \int_{\R_+} ds \int_{\R^d}|u_{x_i x_j}(s,x)|^{p}dx 
 \\
 =\int_{\R_+} ds \int_{\R^d} \left| \E \Big [  \int_s^{\infty}dr
\int_{\R^d} 
 f_{x_i x_j}(r, x+ y + b_r - b_s) \mu_{sr} (dy)
 \Big] \right|^p dx
\\
\le  \E \Big [ \int_{\R_+} ds \int_{\R^d}   \Big|  \int_s^{\infty}dr
\int_{\R^d} 
 f_{x_i x_j}(r, x+ y + b_r - b_s)
 \mu_{sr} (dy) \Big|^p dx \Big].
\end{gather*}
Now in the last line of the previous formula we  change variable in
%$ \int_{\R^d}   \Big|  \int_s^{\infty}dr
%\int_{\R^d} \mu_{sr} (dy)
%\int_{\R^d} f_{x_i x_j}(r, x+ y + b_r - b_s)
%  \Big|^p dx $
the   integral over $\R^d$ with respect to 
the $x$-variable; we  
 obtain
  \begin{equation} \label{fgh}
 \int_{\R_+} ds \int_{\R^d}|u_{x_i x_j}(s,x)|^{p}dx 
 \end{equation}$$ 
 \le    \E \Big [ \int_{\R_+} ds \int_{\R^d}   \Big|  \int_s^{\infty}dr
\int_{\R^d} 
 f_{x_i x_j}(r, z + y + b_r )
 \mu_{sr} (dy) \Big|^p dz \Big].
$$
To estimate the last term we fix  $\omega \in \Omega$ and consider 
the  function $g_{\omega}(t,x)  = f(t, x + b_t(\omega) )$, $(t,x)
\in \R^{d+1}$. 
%Here $b(t) (\omega)$ indicates the random variable $b(t)$
%computed in $\omega$.
The function $g_{\omega}$ is bounded,
with compact support on $\R^{d+1}$ and such
that 
$g_{\omega}(t, \cdot) \in C_0^{\infty}(\R^{d})$, $t \in \R$. 
% One can prove that  Note that 
%We 
%know that there exists a calssical solution for
%\begin{equation} \label{cla}
%v_t + \sum_{i =1}^{p_0} v_{x_i x_i} = g_{\omega},
%\end{equation}
%(cf. Proposition \ref{poco}) and moreover  
%$$
%v(s,x) = \int_s^{\infty}dr
%\int_{\R^d} g_{\omega}(r, x + y ) \mu_{sr}(dy).
%$$

By Lemma \ref{ser} we know  
%$L^p$-regularity theory for \eqref{cla}, 
that there exists $M_0 = M_0
(d,p, p_0)>0$ such that, for any $\omega \in
\Omega$,
\begin{gather*}
%\Big \| \partia_{x_i x_j} R g_{\omega}  \Big \|_{L^p} 
\int_{\R_+} ds \int_{\R^d}   \Big|  \int_s^{\infty}dr
\int_{\R^d}  
 f_{x_i x_j} \big(r, z + y + b_r  (\omega) \big)
 \mu_{sr} (dy) \Big|^p dz
\\
= \int_{\R_+} ds \int_{\R^d}   \Big|  \int_s^{\infty}dr
\int_{\R^d}  
 \partial^2_{x_i x_j}g_{\omega} \big(r, z + y  \big)
 \mu_{sr} (dy) \Big|^p dz
\le M_0^p \,  \|
g_{\omega}\|_{L^p}^p. 
\end{gather*}
Using also \eqref{fgh} we find 
\begin{gather*}
  \int_{\R_+} ds \int_{\R^d}|u_{x_i x_j}(s,x)|^{p}dx
  \le M_0^p\,  \E \Big[ \int_{\R} ds \int_{\R^d}|g_{\omega}(s,x)|^{p}dx \Big]
\\
= M_0^p  \, \E \Big[ \int_{\R} ds \int_{\R^d}|f(s,x + b_s)|^{p}dx \Big]
\\
= M_0^p  \,  \int_{\R} ds \int_{\R^d}|f(s, z)|^{p}dz. 
  \end{gather*}
The proof is complete.  
% 
%\end{proof}

%it is not a well-known fact. Indeed, see for instance 

\section {$L^p$-estimates involving Ornstein-Uhlenbeck operators}
 \label{ou11}

Let $A= (a_{ij})$ be a given real $d \times d$-matrix. We  %concerning with 
%and consider 
consider the
following 
%time-dependent 
Ornstein-Uhlenbeck type operator
\begin{equation*}
L_0 u(t,x) =   \sum_{i,j=1}^d c_{ij}(t) u_{x_i x_j}(t,x) +
 \sum_{i,j=1}^d a_{ij} x_j \, u_{x_i}(t,x)
\end{equation*} 
$$ 
 = \text{Tr}(c(t)D^2_x u(t,x)) + \langle Ax, D_x u(t,x) \rangle,
$$
$ (t,x) \in \R^{d+1}, \;\; u \in  C_{0}^{\infty}(\R^{d+1})$. 
This is a kind of perturbation of  $L$ given in  \eqref{lu} by  the first
order term $\langle Ax, D_x u(t,x) \rangle $ which has linear coefficients.

We will extend Corollary \ref{main} to cover the parabolic equation 
\begin{equation} \label{ou1}
 u_t (t,x) + L_0 u(t,x) = f(t,x) 
\end{equation}  
on $\R^{d+1}$. We will assume Hypothesis \ref{hy1} and also 

\begin{hypothesis} \label{hy2}
 Let $p_0 $ as   in Hypothesis \ref{hy1}. Define 
 $F_{p_0} \simeq \R^{p_0} $ as the linear subspace generated by 
 $\{e_1, \ldots, e_{p_0} \}$. Let $F^{p_0} $ be the linear
subspace generated by 
 $\{e_{p_0 +1}, \ldots, e_{d} \}$ if $p_0 <d$ (when $p_0 =d$, $F^{p_0} = \{
0\}$). We suppose
 that 
\begin{equation}
 A(F_{p_0}) \subset F_{p_0}, \;\;\;\; A(F^{p_0}) \subset F^{p_0}.
  \label{hy22}
\end{equation} 
\end{hypothesis} 
Recall that given a $d \times d$-matrix $B$,  
$\| B\|$ and $Tr(B)$ denote, respectively, the operator norm and the trace
of $B$. In the next result we will use that there exists $\omega >0$ and $\eta >0$ such that
\begin{equation}
\label{om1}
 \| e^{tA} \| \le \eta e^{\omega |t|},\;\;\; t \in \R,
\end{equation} 
where $e^{tA} $ is the exponential matrix of $tA$. Note that the constant $M_0$ below is the same given in \eqref{new}.

%denotes the operator norm of the matrix $A$.
%IN INTRO ??

\begin{theorem}
\label{ouma}  Assume Hypotheses \ref{hy1} and \ref{hy2}.
Let $T>0$ and set $S_T = (-T, T)
\times \R^d  $. Suppose that $u \in C^{\infty}_0 (S_T)$. For any
$p \in (1, \infty)$, 
 $i,j =1, \ldots, p_0$,
% there exists $M_1 = M_1(d,p, p_0, \eta, T \cdot \omega) >0$
% such that 
\begin{align}
\label{esti1}
\|  u_{x_i x_j} \|_{L^p(\R^{d+1})} \le \,   \frac{M_1(T) }
{{\lambda} }\,  \| u_t + L_0 u\|_{L^p(\R^{d+1})}; \; \; %\text{with}
\\ \;  
\nonumber \text{with} \;\; M_1(T) = c (d)  {M_0}  \eta^4  e^{4 T \omega   +  \frac{2T}{p} \, |\text{Tr}(A)| }.
\end{align} 
%where $M_0 = M_0 (d, p, p_0)$, $1 < p < \infty$, is the same constant given   
%in \eqref{new1}.
\end{theorem}
\begin{proof} We fix $T>0$ and 
use a change of variable similar to that  used in page 100 of \cite{DL}. 
  Define $v(t,y) = u(t, e^{tA} y)$, $(t,y) \in
\R^{d+1}$. We have $v \in C_0^{\infty} (\R^{d+1})$,  $u(t,x) = v(t, e^{-tA} x)$ 
and 
\begin{gather*}
 u_t(t,x) + L_0 u(t,x)
 \\ =  v_t(t,e^{-tA} x) - \langle D_y v
 (t, e^{-tA} x) , A e^{-tA} x\rangle  + {\text Tr} \big( e^{-tA}c(t)
 e^{-tA^*} D^2_y  v(t, e^{-tA} x) \big)
 \\
 + \langle D_y v
 (t, e^{-tA} x) , A e^{-tA} x\rangle
\\ 
= v_t(t,e^{-tA} x)    + {\text Tr} \big( e^{-tA}c(t)
 e^{-tA^*} D^2_y  v(t, e^{-tA} x) \big).
\end{gather*}
It follows that
\begin{equation}
\label{chi}
 u_t(t,e^{tA} y) + L_0 u(t,e^{tA} y)
= v_t(t,y)    + {\text Tr} \big( e^{-tA}c(t)
 e^{-tA^*} D^2_y  v(t, y) \big).
\end{equation} 
Now we have to check
%that 
%$e^{-tA}c(t)
% e^{-tA^*}$ verifies 
Hypothesis \ref{hy1}. We first define $c_0 (t)$, $t \in
\R$,
\begin{equation}
\label{coii}
c_0(t) = e^{-tA}c(t)
 e^{-tA^*},\;\;\; t \in [-T,T],
\end{equation}
$$ c_0(t) = e^{- TA}c(T)
 e^{- TA^*}, \;\; t \ge T,\;\;  c_0(t) = e^{ TA}c(-T)
 e^{ TA^*}, \;\; t \le -T.
$$
Since $v \in C_0^{\infty} (S_T)$ we have on $\R^{d+1}$
$$ 
v_t(t,y)    + {\text Tr} \big( e^{-tA}c(t)
 e^{-tA^*} D^2_y  v(t, y) \big)=
  v_t(t, y)    + {\text Tr} 
  \big(c_0(t) D^2_y  v(t, y)
\big)
$$
and so it is enough to check that  $c_0(t)$ verifies \eqref{nondeg7}. Moreover, by \eqref{coii} it is enough to verify 
\eqref{nondeg7} for $t \in
[-T,T]$. We have
$$
\langle c_0(t) \xi, \xi 
\rangle = 
\langle c(t) e^{-tA^*} \xi, e^{-tA^*}\xi 
\rangle  \ge {\lambda}  \, |I_0 e^{-tA^*}\xi |^2.
$$
By \eqref{hy22} we deduce that $F_{p_0}$ and $F^{p_0}$ are both invariant for
$A^*$. It follows easily that 
\begin{equation} \label{inv}
I_0 e^{sA^*}\xi  =  e^{sA^*} I_0 \xi,\;\;\; \xi \in \R^d, \;\; s\in \R. 
\end{equation} 
Using this fact we find for $t \in [-T, T]$, $\xi \in \R^d$, 
$$
|I_0 \xi|^2 = |I_0 e^{tA^*} e^{-tA^*}  \xi|^2 =
|e^{tA^*} I_0 e^{-tA^*}  \xi|^2 \le \eta^ 2 e^{2 T \omega} \, 
|I_0 e^{-tA^*}  \xi|^2 
$$
and so
$$
{\lambda}  |I_0 \xi|^2 \le {\lambda}  \eta^2 e^{2 T \omega } \, 
|I_0 e^{-tA^*}  \xi|^2  \le \eta^2 e^{2 T \omega} \, 
 \langle c_0(t) \xi, \xi 
\rangle, 
$$
which implies ${\lambda} \eta^{-2}  e^{-2 T \omega } \, |I_0 \xi|^2 \le 
 \langle c_0(t) \xi, \xi 
\rangle $.
 By Corollary \ref{main} and \eqref{chi} we get, for any $i, j =1, 
\ldots p_0$,
\begin{equation} \label{vino}
 \| v_{y_i y_j}  \|_{L^p} 
 = \| \langle D^2_y v(\cdot) e_i, e_j \rangle \|_{L^p}  \le 
\frac{M_0 \, \eta^2  e^{2 T \omega}}{{\lambda} }
 \|  v_t  \, + \, \text{Tr} (c_0(t)D^2_y v) \|_{L^p} 
\end{equation}
$$
=
\frac{M_0 \, \eta^2  e^{2 T \omega}}{{\lambda} }
 \|  u_t(\cdot ,e^{\cdot \, A} \cdot) + L_0 u(\cdot ,e^{\cdot \, A} \cdot
)\|_{L^p} 
\le \frac{M_0 \eta^2 e^{2 T \omega }}{{\lambda} } e^{\frac{T}{p} \, |\text{Tr}(A)|}
 \, \|  u_t + L_0 u\|_{L^p}.
$$
Note that
\begin{gather*}
\langle D^2_y v (t,y) I_0 e_i, I_0 e_j \rangle  = \langle D^2_y v(t,y) e_i, e_j
\rangle  =
\langle e^{tA^*} D^2_x u(t, e^{tA} y)
 e^{tA}  e_i, e_j \rangle
\end{gather*}
and so $I_0  D^2_y v(t,y) I_0 = e^{tA^*} I_0 D^2_x u(t, e^{tA} y)
 I_0 e^{tA}$, $t \in \R,$ $y \in \R^d$. Indicating by $\R^{p_0} \otimes
\R^{p_0}$  the space of all real $p_0 \times p_0$-matrices, we find
\begin{gather*}
\| I_0  D^2_y v \, I_0 \|_{L^p (\R^{d+1}; \R^{p_0} \otimes \R^{p_0})}
\ge  e^{- \frac{T}{p} \, |\text{Tr}(A)|}  \, \| 
 e^{ \cdot \, A^*} I_0 D^2_x u \, 
 I_0 e^{\cdot \, A} \|_{L^p (\R^{d+1}; \R^{p_0} \otimes \R^{p_0})}.
\end{gather*}
Since, for $(t,x) \in \R^{d+1}$, 
$$
\|I_0 D^2_x u (t,x)\, 
 I_0 \| \le  \eta^2 e^{2 T  \omega }
\,  \| e^{ t \, A^*} I_0 D^2_x u (t,x)\, 
 I_0 e^{t A} \|
$$
by \eqref{vino} we deduce
$$
\| I_0  D^2_x u \, I_0 \|_{L^p (\R^{d+1}; \R^{p_0} \otimes \R^{p_0})}
 \le c (d) \, \frac{M_0}{{\lambda} } \, \eta^4 e^{4 T  \omega} \, e^{ \frac{2T}{p} \,
|\text{Tr}(A)|}  \, \|  u_t + L_0 u\|_{L^p}
$$
which gives \eqref{esti1}.
%with 
%\begin{equation}
%\label{cost}
%M_1 = c (d) \, {M_0} \, \eta^4  e^{4 T \omega } \, e^{ \frac{2T}%{p} \, |\text{Tr}(A)|}.
%\end{equation} 
The proof is complete.
 
\end{proof}
   
\begin{example} \label{exa1} {\em The equation  
\begin{equation} \label{eq23}
u_t(t,x,y) + (1+ e^t) u_{xx}(t,x,y) +
 t u_{xy}(t,x,y) + t^2 u_{yy}(t,x,y)   + y u_y (t,x,y) = f(t,x,y),
\end{equation}
 $(t,x,y) \in \R^3$, verifies the assumptions of Theorem \ref{ouma} with $p_0 =1$ and
so estimate \eqref{esti1} holds for $u_{xx} $. }
\end{example} 

\begin{remark} {\em
Assumption \eqref{hy22} does not hold for the degenerate hypoelliptic
%Ornstein-Uhlenbeck 
operators considered in \cite{BCLP1}. To see this let us consider 
the following   classical example of hypoelliptic operator (cf.
 \cite{H} and \cite{Kol})
\begin{equation}
\label{ko}
u_t(t,x,y) + u_{xx}(t,x,y) + x u_y (t,x,y)
  = f(t,x,y),
\end{equation}
 $(t,x,y) \in \R^3$. In this case $p_0 =1$ and 
$ A = \begin{pmatrix}
0 \;  0
\\
1\; 0
   \end{pmatrix} $. It is clear that \eqref{hy22}
does not hold in this case.  Indeed we can not recover the $L^p$-estimates
in \cite{BCLP1}. }
\end{remark}

As an application of the previous theorem we obtain elliptic estimates for
non-degenerate Ornstein-Uhlenbeck operators ${\cal A}$. These
estimates have been  first proved in \cite{MPRS}. Differently with respect to
\cite{MPRS} in the next result we can show the  explicit dependence of the
constant $C_1$ in \eqref{esti11} from the ellipticity constant ${\lambda} $. 
   
Let 
\begin{equation} \label{ou22}
{\cal A} u(x)      
 = \text{Tr}(Q \,  D^2 u(x)) + \langle Ax, Du(x) \rangle,
\end{equation}
$x \in \R^{d},$ $ u \in  C_{0}^{\infty}(\R^{d})$, where $A$ is a $d \times d $
matrix and  $Q $ is a symmetric
positive define $d \times d $-matrix such that  
\begin{equation} \label{frt}
 \langle Q \xi , \xi \rangle \ge {\lambda}  |\xi|^2,\;\;\; \xi \in \R^d,
\end{equation}  
for some ${\lambda}  >0.$

\begin{corollary}
\label{ouu} Let us consider \eqref{ou22} under assumption \eqref{frt}.
% There exists a positive constant $M_2$ $ = M_2 (d,p, \| A\|)$ 
% such that $M_2 (d,p, \cdot)$ is increasing and, 
For any   $w \in C^{\infty}_0
(\R^d)$, 
$p \in (1, \infty)$, 
 $i,j =1, \ldots, d$, we have (the constant $M_1(1)$ is given in  \eqref{esti1})
\begin{equation}
\label{esti11}
\|  w_{x_i x_j} \|_{L^p(\R^{d})} \le \, \frac{c(p) \, M_1(1)}{{\lambda}}   \big(
 \| {\cal A} w\|_{L^p(\R^{d})} + \| w\|_{L^p(\R^d)} \big).
\end{equation} 
% where $M_1(1)$ is given in  \eqref{esti1}.
 % with $C_1 = \frac{M_2}{{\lambda} }$.
\end{corollary}
\begin{proof} We will deduce \eqref{esti11} from \eqref{esti1} 
 in $S_1 = (-1, 1) \times \R^d$ with $p_0 =d$.

Let $\psi \in C^{\infty}_0
(-1,1)$ with $\int_{-1}^1 \psi(t) dt >0 $. We define, similarly to  Section
1.3 of \cite{BCLP1},
$$
u(t,x) = \psi(t) w(x).
$$
Since $u_t + L_0 u = \psi'(t) w(x) + \psi (t) {\cal A} w(x)$,
applying \eqref{esti1} to $u$ we easily get \eqref{esti11}.
 
\end{proof}

%%%%%%%%%%%%%%%%%%%%%%%% referenc.tex %%%%%%%%%%%%%%%%%%%%%%%%%%%%%%
% sample references
% %
% Use this file as a template for your own input.
%
%%%%%%%%%%%%%%%%%%%%%%%% Springer-Verlag %%%%%%%%%%%%%%%%%%%%%%%%%%
%
% BibTeX users please use
% \bibliographystyle{}
% \bibliography{}
%
%\biblstarthook{References may be \textit{cited} in the text either by number %(preferred) or by author/year.\footnote{Make sure that all references from %the list are cited in the text. Those not cited should be moved to a %separate \textit{Further Reading} section or chapter.} The reference list %should ideally be \textit{sorted} in alphabetical order -- even if reference %numbers are used for the their citation in the text. If there are several %works by the same author, the following order should be used: 
%\begin{enumerate}
%\item all works by the author alone, ordered chronologically by year of %publication
%\item all works by the author with a coauthor, ordered alphabetically by %coauthor
%\item all works by the author with several coauthors, ordered %chronologically by year of publication.
%\end{enumerate}
%}

%}
\end{document}